\documentclass[reqno,11pt]{amsart}
\usepackage{latexsym}
\usepackage{amsfonts}
\usepackage{amssymb}
\usepackage{mathrsfs}
\usepackage[all]{xy}
\usepackage{bbm}
\usepackage{color}

\newtheorem{theorem}{Theorem}[section]

\newtheorem{lemma}[theorem]{Lemma}
\newtheorem{conj}[theorem]{Conjecture}

\theoremstyle{definition}

\newtheorem{remark}[theorem]{Remark}
\numberwithin{equation}{section}

\newcommand{\RR}{\mathbb{R}}
\newcommand{\NN}{\mathbb{N}}
\newcommand{\QQ}{\mathbb{Q}}
\newcommand{\ZZ}{\mathbb{Z}}
\newcommand{\mt}{\mapsto}

\newcommand{\cV}{\mathcal{V}}

\newcommand{\cN}{\mathcal{N}}
\newcommand{\cW}{\mathcal{W}}

\newcommand{\sB}{\mathscr{B}}
\newcommand{\Bad}{\mathbf{Bad}}
\newcommand{\bx}{\mathbf{x}}

\newcommand{\bv}{\mathbf{v}}
\newcommand{\br}{\mathbf{r}}
\newcommand{\bw}{\mathbf{w}}

\newcommand{\diag}{\mathrm{diag}}
\newcommand{\dist}{\mathrm{dist}}

\newcommand{\SL}{\mathrm{SL}}
\newcommand{\Ad}{\mathrm{Ad}}

\newcommand {\ignore}[1]  {}

\newif\ifdraft\drafttrue
\draftfalse

\newcommand\eq[2]{{\ifdraft{\ \tt [#1]}\else\ignorespaces\fi}\begin{equation}\label{eq:#1}{#2}\end{equation}}

\newcommand {\equ}[1]     {\eqref{eq:#1}}
\newcommand{\ggm}{G/\Gamma}

\newcommand\hd{Hausdorff dimension}
\newcommand\ba{badly approximable}
\newcommand\hs{homogeneous space}

\newcommand{\x}{{\bf x}}
\newcommand{\rr}{{\bf r}}
\renewcommand{\emptyset}{\varnothing}
\renewcommand{\setminus}{\smallsetminus}

\begin{document}

\title{Bounded orbits of diagonalizable flows on $ \SL_3(\RR)/\SL_3(\ZZ)$}

\author{Jinpeng An}
\address{LMAM, School of Mathematical Sciences, Peking University, Beijing, 100871, China}
\email{anjinpeng@gmail.com}

\author{Lifan Guan}
\address{School of Mathematical Sciences, Peking University, Beijing, 100871, China}
\email{guanlifan@gmail.com}

\author{Dmitry Kleinbock}
\address{Department of Mathematics, Brandeis University, Waltham MA 02454-9110, USA}
\email{kleinboc@brandeis.edu}

%\date{}
\thanks{Research supported by NSFC grant 11322101 (J.A.) and by NSF grant DMS-1101320 (D.K.)}

\begin{abstract}
We prove that for any countably many one-parameter diagonalizable subgroups $F_n$ of $\SL_3(\RR)$, the set of $\Lambda\in\SL_3(\RR)/\SL_3(\ZZ)$ such that all the orbits $F_n\Lambda$ are bounded has full Hausdorff dimension.
\end{abstract}

\maketitle

\section{Introduction}

The problems considered in this paper are motivated by concepts coming from number theory, in particular, by the notion of \ba\ pairs\footnote{
%The concepts and  results introduced in the historical account below are much more general \textcolor{red}{(About English usage: do we mean the concepts and  results are much more general than what are introduced below?)} -- s
Similarly to $(x,y)\in \RR^2$ one can treat $m\times n$ matrices (systems of linear forms), which would correspond  to flows on \hs\ $ \SL_d(\RR)/\SL_d(\ZZ)$ for $d = m+n$. However, in the present paper we are able to establish our theorems  only for the case $d = 3$; thus we chose to concentrate on the lower-dimensional case for simplicity of the exposition. See \S\ref{S:7} for generalizations. }   of real numbers. Recall that  $(x,y)\in\RR^2$ is called {\sl badly approximable\/}  (notation: $(x,y)	\in\Bad$) if
\eq{def ba}{\max\{|qx - p|, |qy - r|\} \geq \frac{c}{q^{1/2}} }
for some $c=c(x,y)>0$ and for any $(p,r) \in \mathbb{Z}^{2}$, $q\in\NN$. The set $ \Bad$ has Lebesgue measure zero.
On the other hand, it is
{\sl thick}, that is,
the \hd\ of its intersection with any non-empty open subset of $\RR^2$  is equal to the dimension of the ambient space.
 The proof of the latter property, due to %W.\
 Schmidt \cite{Sc1}, consisted in showing that $ \Bad$ is a winning set for a certain game introduced by Schmidt in that same paper. Schmidt showed that winning sets are thick and have remarkable countable intersection properties (see \S\ref{alphabeta} for more detail); thus the winning property is considerably stronger than thickness.

In 1986 %S.G.\
Dani \cite{Da1} exhibited an important interpretation of badly approximable vectors via dynamics on the space of lattices. Namely,
let \eq{sld}{G=\SL_{3}(\RR),\ \Gamma=\SL_{3}(\ZZ),\ X  = \ggm;} elements   of the \hs\  $X$ can be viewed as unimodular lattices in $\RR^{3}$,
via the identification of  $\Lambda = g\ZZ^3$ with $g\Gamma$.
Consider \eq{defu0}{U_0 = \{u_{x,y} : (x,y)\in\RR^2\}, \quad \text{where } u_{x,y} =\begin{pmatrix}
  1&& x\\ &1&y\\ &&1
\end{pmatrix}\in G.}
Capitalizing on earlier observations by Davenport and Schmidt, Dani showed that $(x,y)\in\Bad$ if and only if the trajectory $F^+ u_{x,y}\ZZ^{3}$, where
 \eq{sldgt}{F^+ = \{g_t: t \ge 0\}\text{ and }g_t = \diag(e^{t/2},e^{t/2}, e^{-t} )\,,}
is bounded in $X$. Also it is easy to see that the group $U_0$ is the {\sl expanding horospherical subgroup\/} of $G$ relative to $F^+$, that is, defined by
\eq{defehs}{H(F^+) := \{h\in G:\lim_{t\to+\infty}g_t^{-1}hg_t=e\}\,.}
This makes it possible to conclude that when $G$, $\Gamma$ are as in  \equ{sld} and $F^+$ is as in  \equ{sldgt},  the set  \eq{einfinity}{E(F^+) := \{\Lambda\in \ggm: F^+\Lambda\text{ is bounded} \} } is thick. Note that the above set has Haar measure zero in view of the ergodicity of the $g_t$-action on $\ggm$.

Now let us keep  $G$, $\Gamma$ and $X$ as in  \equ{sld}, and take a more general one-parameter subsemigroup $F^+$ of $G$.  %G.A.\
Margulis
conjectured\footnote{The conjecture of Margulis was stated in \cite{Ma} and settled in \cite{KM} in much bigger generality, for any Lie group $G$,
any lattice $\Gamma \subset G$ and any non-quasiunipotent subgroup $\{g_t\}$ of $G$.}
%\footnote{The scope of the conjecture was much more general, not restricted to the case \equ{sld}.}
 in \cite{Ma} that the set \equ{einfinity} is thick whenever $g_1$ is {\sl not quasiunipotent}, that is, $\Ad\, g_1$ has at least one eigenvalue with modulus different from $1$. (The quasiunipotent case is drastically different due to Ratner's Theorem.) This conjecture was settled in a subsequent work of Margulis and the third-named author \cite{KM}.

A Diophantine interpretation of actions of generic diagonal subgroups of $G$ was first suggested in \cite{Kl}. Namely, take $ \lambda,\mu$ with \eq{condr}{\lambda,\mu \ge 0\text{ and }\lambda + \mu = 1\,,} and consider
 \eq{sldgtweights}{g_t = \diag(e^{\lambda t},e^{\mu t}, e^{-t} )\text{ and } F^+_{\lambda,\mu} = \{g_t: t \ge 0\}\,.}
 It was observed in \cite{Kl} that  $ u_{x,y}\ZZ^{3}\in E( F^+_{\lambda,\mu} )$ if and only if $(x,y)$ is {\sl badly approximable  with  weights} $\lambda,\mu$, that is, there is $c=c(x,y)>0$ such that $$\max\{ q^{\lambda}|qx-p|, q^{\mu}|qy-r|\}\ge c$$ for all   $(p,r) \in \mathbb{Z}^{2}$, $q\in
\mathbb{N}$. We will denote the set of $(x,y)$ with this property by $\Bad(\lambda,\mu)$. The classical case corresponds to the uniform choice of weights $\lambda = \mu = \frac12$. The sets $\Bad(\lambda,\mu)$ are known to be Lebesgue null \cite{Sc0} and thick \cite{PV, KW2}.

During recent years there has been a rapid progress in studying %sets  $\Bad(\lambda,\mu)$
%and
the intersection properties of those sets. A quite general form of a conjecture of Schmidt \cite{Sc3} asserting that $$\Bad(\tfrac13,\tfrac23) \cap \Bad(\tfrac23,\tfrac13) \ne\varnothing$$ was proved by Badziahin, Pollington and Velani \cite{BPV}. Then it was shown by the first-named author \cite{An1, An2} that for any $\lambda,\mu$ satisfying  \equ{condr}, $\Bad(\lambda,\mu)$ is an $\alpha$-winning subset of $\RR^2$ for some $\alpha$ independent of $(\lambda,\mu)$; see also \cite{ABV, NS} for  stronger results. Consequently, for any  countable collection $\{(\lambda_n,\mu_n) : n\in \NN\}$ of weight vectors, the intersection $\bigcap_n\Bad(\lambda_n,\mu_n)$ is thick (this was proved in \cite{BPV} under an additional restriction on the weight vectors). In view of the aforementioned correspondence, this proves that the intersection
 \eq{thickint}{
\bigcap_nE( F^+_{\lambda_n,\mu_n} )
}
is non-empty: indeed, it contains lattices of the form $ u_{x,y}\ZZ^{3}$ where $(x,y)\in \bigcap_n\Bad(\lambda_n,\mu_n)$.

\medskip

However, the above statement is not strong enough to prove that the set \equ{thickint} is a thick subset of $X$. Indeed, the difficulty comes from the fact that whenever $(\lambda,\mu)\ne ( \frac12, \frac12)$, the  expanding horospherical subgroup $ H(F^+_{\lambda,\mu})$ is strictly larger than  $U_0$; in fact, when $\lambda > \mu$ it is equal to the three-dimensional (Heisenberg) upper-triangular group
\eq{defu}{ U := \{u_{x,y,z} : (x,y,z)\in\RR^3\},\text{  where }u_{x,y,z}:=\begin{pmatrix}
  1&z&x \\&1&y\\&&1
\end{pmatrix}.}
The main goal of the present paper is to overcome this difficulty. For fixed $(\lambda,\mu)$ as in \equ{condr} with  $\lambda  >\mu$, we will show (see Theorem \ref{T:HAW_U})   that the set $$\big\{u\in U : u\ZZ^3\in E(F^+_{\lambda,\mu})\big\}$$ is a winning subset of $U$. This will make it possible to establish a certain winning property of the sets  $E(F^+_{\lambda,\mu}$) which will guarantee the thickness of their countable intersection.

Moreover, we are able to move beyond the diagonal case, and consider arbitrary diagonalizable one-parameter subgroups of $G$. In the three theorems below $G$, $\Gamma$ and $X$ are as in \equ{sld}. The following is one of our main results:
%It would be natural to attempt to

\begin{theorem}\label{T:dim}
Let $\{F_n:n\in\NN\}$ be one-parameter diagonalizable subgroups of $G$. Then the set
\begin{equation}\label{E:set1}
%\{x\in G/\Gamma : F_nx \text{ is bounded for every } n\in\NN\}
\bigcap_nE( F_n )
\end{equation}
is a thick subset of $X$.
\end{theorem}

Note that here we are considering the orbits of subgroups $F_n$ rather than of their non-negative parts (semigroups  $F^+_n$).

\medskip

In order to prove Theorem \ref{T:dim}, we are going to use a variant of the winning property -- so-called Hyperplane Absolute Winning (HAW) \cite{BFKRW}. It has many advantages over the traditional version, one of which is a possibility to define the game on smooth manifolds in an invariant way, as demonstrated recently by the third-named author and B.\ Weiss \cite{KW3}. The history, definitions and properties of the hyperplane modification of Schmidt's game are described in \S\S\ref{haw}--\ref{hawmfld}. Theorem \ref{T:dim}  will be deduced from the following:

\begin{theorem}\label{T:HAW}
Let $F=\{g_t:t\in\RR\}$ be a one-parameter diagonalizable subgroup of $G$, and let $F^+=\{g_t:t\ge0\}$. Then the set $E(F^+)$ is HAW on $X$.
\end{theorem}

This in particular solves the $G=\SL_3(\RR)$ case of Question 8.1 in \cite{Kl}.

Regarding the expanding horospherical subgroup, we will also prove:

\begin{theorem}\label{T:HAW-expand}
Let $F^+$ be as in Theorem \ref{T:HAW} and let $H=H(F^+)$ be as in \equ{defehs}. Then for any $\Lambda\in X$, the set
$$\big\{h\in H :  h\Lambda \in E(F^+)\big\}$$
is HAW on $H$.
\end{theorem}

The paper is organized as follows. In \S\ref{S:2}, we briefly recall the definitions and basic properties of Schmidt's game and its variants -- the hyperplane absolute game and the hyperplane potential game. The latter game, being introduced in \cite{FSU}, has the same winning sets as the hyperplane absolute game. It turns out that the potential game is an effective tool for proving the winning property for the absolute game. We also deduce Theorem \ref{T:dim} from Theorem \ref{T:HAW}. In \S\ref{S:3}, we highlight a special case of Theorem \ref{T:HAW-expand}, that is, Theorem \ref{T:HAW_U} below. The proof of Theorem \ref{T:HAW_U}, which forms the most technical part of this paper, is given in \S\S\ref{S:3}--\ref{S:pf_main}. In \S\ref{S:6} we complete the proofs of Theorems \ref{T:HAW} and \ref{T:HAW-expand} by using Theorem \ref{T:HAW_U}, and in the last section of the paper discuss various extensions and open questions.

\section{Preliminaries on Schmidt games}\label{S:2}

\subsection{Schmidt's $(\alpha,\beta)$-game}\label{alphabeta}

We first recall Schmidt's $(\alpha,\beta)$-game introduced in \cite{Sc1}. It
involves two parameters $\alpha,\beta\in(0,1)$  and is played by two players Alice and Bob on a complete metric space $(X,\dist)$ with a target set $S\subset X$. Bob starts the game by choosing a closed ball $B_0=B(x_0,\rho_0)$ in $X$ with center $x_0$ and radius $\rho_0$. After Bob chooses a closed ball $B_i = B(x_i,{\rho}_i) $, Alice chooses  $A_i = B(x_i', {\rho}_i')$ with $${\rho}'_i=\alpha {\rho}_i\text{ and }
\dist(x_i, x'_i) \leq (1-\alpha){\rho}_i\,,$$
 and then Bob chooses  $B_{i+1} = B(x_{i+1},{\rho}_{i+1}) $ with $${\rho}_{i+1}=\beta {\rho}'_{i} \text{ and }  \dist(x_{i+1}, x'_i) \leq (1-\beta){\rho}_i \,,$$
etc. This implies that
the balls are nested:
$$
B_0 \supset A_0 \supset B_1 \supset\cdots ;$$
%, where $\rho(\cdot)$ denotes the radius of a ball.
Alice wins the game if the unique point  $\bigcap_{i=0}^\infty A_i=\bigcap_{i=0}^\infty B_i$ belongs to $S$, and Bob wins otherwise. The set $S$ is \textsl{$(\alpha,\beta)$-winning} if Alice has a winning strategy, is \textsl{$\alpha$-winning} if it is $(\alpha,\beta)$-winning for any $\beta\in(0,1)$, and is \textsl{winning} if it is $\alpha$-winning for some $\alpha$. Schmidt \cite{Sc1} proved that:

\begin{itemize}
\item[$\bullet$]  winning subsets of
Riemannian manifolds
are thick;
\item[$\bullet$]  if $S$ is $\alpha$-winning and $\varphi:X\to X$ is bi-Lipschitz, then $\varphi(S)$ is   $\alpha'$-winning, where $\alpha'$ depends on $\alpha$ and the bi-Lipschitz constant of $\varphi$;
\item[$\bullet$]  a countable intersection of $\alpha$-winning sets is again $\alpha$-winning.
\end{itemize}

As such, Schmidt's game has been a powerful tool for proving %full dimensionality (and even nonemptiness)
thickness of intersections of certain countable families of sets, see e.g.\ \cite{An1, An2,  BBFKW, BFK, BFKRW, Da1, Da2, Da3}. However, for a fixed $\alpha$, the class of $\alpha$-winning subsets of a Riemannian manifold depends on the choice of the metric, and is not preserved by diffeomorphisms. This makes it difficult to prove statements like Theorem \ref{T:dim} using the $(\alpha,\beta)$-game directly.

\subsection{Hyperplane absolute game on $\RR^d$}\label{haw}

Inspired by ideas of McMullen \cite{Mc}, the hyperplane absolute game on the Euclidean space $\RR^d$ was introduced in \cite{BFKRW}.
It has the advantage that the family of its winning sets is preserved by $C^1$ diffeomorphisms.
Let $S \subset \RR^d$ be a target set and let
$\beta \in \left(0, \frac13 \right)$.
As before Bob begins by choosing a closed ball $B_0$ of radius ${\rho}_0$.
For an affine hyperplane $L\subset\RR^d$ and $\rho>0$, we denote the ${\rho}$-neighborhood of $L$ by
$$L^{({\rho})}:=\{\bx\in\RR^d:\mathrm{dist}(\bx,L)<{\rho}\}.$$
Now, after Bob chooses a closed ball $B_i$ of radius ${\rho}_i$, Alice chooses a hyperplane neighborhood $L_i^{({\rho}'_i)}$ with ${\rho}'_i\le\beta {\rho}_i$,
and then Bob chooses a closed ball $B_{i+1}\subset B_i\setminus L_i^{({\rho}'_i)}$ of radius ${\rho}_{i+1}\ge\beta {\rho}_i$. Alice wins the game if and only if
$$\bigcap_{i=0}^\infty B_i\cap S\ne\emptyset.$$ The set $S$ is \textsl{$\beta$-hyperplane absolute
winning} (\textsl{$\beta$-HAW} for short)
if Alice has a winning strategy, and is \textsl{hyperplane absolute winning} (\textsl{HAW} for short)
if it is $\beta$-HAW for any $\beta\in(0,\frac{1}{3})$.

\begin{lemma}[\cite{BFKRW}]\label{L:HAW-Rd}
\begin{itemize}
  \item[(i)] HAW subsets   are winning, and hence thick.
  \item[(ii)] A countable intersection of HAW subsets   is again HAW.
  \item[(iii)]   The  image of an HAW set under a $C^1$ diffeomorphism $\RR^d \to \RR^d$ is HAW.
\end{itemize}
\end{lemma}

\subsection{HAW subsets of a manifold}\label{hawmfld}

The notion of HAW sets has been extended to subsets of $C^1$ manifolds in \cite{KW3}.
This is done in two steps. First one defines the absolute hyperplane game on an open subset $W
\subset \RR^d$. It is defined just as the absolute hyperplane game on
$\RR^d$, except for requiring that Bob's first move $B_0$ be
contained in $W$. If Alice has a winning strategy, we say that $S$ is
{\sl HAW on $W$}.
 Now let $M$ be a $d$-dimensional $C^1$ manifold, and let $\{(U_\alpha, \varphi_\alpha)\}$ be a $C^1$ atlas, that is, $\{U_\alpha\}$ is an open cover of $M$, and each $\varphi_\alpha$ is a $C^1$ diffeomorphism from $U_\alpha$ onto the open subset $\varphi_\alpha(U_\alpha)$ of $\RR^d$. A subset $S\subset M$ is said to be \textsl{HAW on $M$} if for each $\alpha$, $\varphi_\alpha(S\cap U_\alpha)$ is HAW on $\varphi_\alpha(U_\alpha)$. Note that Lemma \ref{L:HAW-Rd} (iii) implies that the definition is independent of the choice of the atlas (see \cite{KW3} for details).

\begin{lemma}[\cite{KW3}]\label{L:HAW-mnfd}
\begin{itemize}
\item[(i)]  HAW subsets of a $C^1$ manifold are thick.
\item[(ii)]  A countable intersection of HAW subsets of a $C^1$ manifold is again HAW.
\item[(iii)] Let $\varphi:M\to N$ be a diffeomorphism between $C^1$ manifolds, and let $S\subset M$ be an HAW subset of $M$. Then $\varphi(S)$ is an HAW subset of $N$.
\item[(iv)] Let $M$ be a $C^1$ manifold with an open cover $\{U_\alpha\}$. Then a subset $S\subset M$ is HAW on $M$ if and only if $S\cap U_\alpha$ is HAW on $U_\alpha$ for each $\alpha$.
\item[(v)] Let $M, N$ be $C^1$ manifolds, and let $S\subset M$ be an HAW subset of $M$. Then $S\times N$ is an HAW subset of $M\times N$.
\end{itemize}
\end{lemma}

\begin{proof}
(i)--(iii) appeared as \cite[Proposition 3.5]{KW3} and are clear from Lemma \ref{L:HAW-Rd}. (iv) is clear from the definition. (v) is proved in \cite[Proof of Theorem 3.6(a)]{KW3}
\end{proof}

\subsection{Proof of Theorem \ref{T:dim} assuming Theorem \ref{T:HAW}} We can now see that Theorem \ref{T:dim} follows from Theorem \ref{T:HAW}. In fact, each $F_n$ in Theorem \ref{T:dim} can be divided into the union of two subsemigroups of the form $F^+$ appeared in Theorem \ref{T:HAW}. So the set \eqref{E:set1} is a countable intersection of sets of the form $E(F_n^+)$. Thus the thickness of the set \eqref{E:set1} follows from the HAW property of the sets $E(F_n^+)$ and parts (i), (ii) of Lemma \ref{L:HAW-mnfd}.

\subsection{Hyperplane potential game}

Finally, we recall the hyperplane potential game introduced in \cite{FSU}. Being played  on $\RR^d$, it has the same winning sets as the hyperplane absolute game. This allows one to prove the HAW property of a set $S\subset\RR^d$ by showing that it is winning for the hyperplane potential game (see \cite{NS}).

Let $S\subset\RR^d$ be a target set, and let $\beta\in(0,1)$, $\gamma>0$. The \textsl{$(\beta,\gamma)$-hyperplane potential game} is defined as follows: Bob begins by choosing a closed ball $B_0\subset\RR^d$. After Bob chooses a closed ball $B_i$ of radius ${\rho}_i$, Alice chooses a countable family of hyperplane neighborhoods $\{L_{i,k}^{({\rho}_{i,k})} : k\in\NN\}$ such that
$$\sum_{k=1}^\infty {\rho}_{i,k}^\gamma\le(\beta {\rho}_i)^\gamma,$$
and then Bob chooses a closed ball $B_{i+1}\subset B_i$ of radius $\rho_{i+1}\ge\beta \rho_i$. Alice wins the game if and only if
$$\bigcap_{i=0}^\infty B_i\cap\Big(S\cup\bigcup_{i=0}^\infty\bigcup_{k=1}^\infty L_{i,k}^{(\rho_{i,k})}\Big)\ne\emptyset.$$
The set $S$ is \textsl{$(\beta,\gamma)$-hyperplane potential winning} (\textsl{$(\beta,\gamma)$-HPW} for short) if Alice has a winning
strategy, and is \textsl{hyperplane potential winning} (\textsl{HPW} for short) if it is $(\beta,\gamma)$-HPW for any $\beta\in(0,1)$ and $\gamma>0$. The following lemma is a special case of \cite[Theorem C.8]{FSU}.

\begin{lemma}\label{L:HPW}
A subset $S\subset\RR^d$ is HPW if and only if it is HAW.
\end{lemma}

\begin{remark}\label{R:to0}
Similar to \cite[Remark 3.2]{KW3} for the HAW property, when proving a set $S\subset\RR^d$ is HPW, we may assume ${\rho}_i\to0$ and prove that Alice can force the unique point $\bx_\infty$ in $\bigcap_{i=0}^\infty B_i$ to be in $S\cup\bigcup_{i=0}^\infty\bigcup_{k=1}^\infty L_{i,k}^{({\rho}_{i,k})}$. In fact, in the first round, Alice can choose $L_{0,k}^{({\rho}_{0,k})}$ such that $\bigcup_{k=1}^\infty L_{0,k}$ is dense in $\RR^d$. This already guarantees a win for Alice if ${\rho}_i\not\to0$. Then in the subsequent rounds, Alice can relabel $B_i$ as $B_{i-1}$ and care only about the case of ${\rho}_i\to0$.
\end{remark}

\section{A special case of Theorem \ref{T:HAW-expand}}\label{S:3}

Both Theorem \ref{T:HAW} and Theorem \ref{T:HAW-expand} will be deduced from the following result:

\begin{theorem}\label{T:HAW_U}
%${r_1}\ge{r_2}\ge0$, ${r_1}+{r_2}=1$.
Let  $(\lambda,\mu)$ be as in  \equ{condr} with $\lambda \ge \mu$,
$F^+_{\lambda,\mu}$ as in  \equ{sldgtweights} and  $U$ as in \equ{defu}.
Then the set
$$\{u\in U : u\Gamma \in E(F^+_{\lambda,\mu})\}$$
is HAW on $U$.
\end{theorem}

Note that if $\lambda > \mu$, then $U$ is the expanding horospherical subgroup for $F^+_{\lambda, \mu}$. Hence, in this case, Theorem \ref{T:HAW_U} is a special case of Theorem \ref{T:HAW-expand}.

\begin{remark}\label{R:special}
In the degenerate case $\lambda = \mu =1/2$, Theorem \ref{T:HAW_U} is a consequence of the HAW property of $\Bad$ proved in \cite{BFKRW}. In fact, in this case, $F^+_{\lambda,\mu}$ is the semigroup $F^+$ defined in \equ{sldgt}. In view of the commutativity of $F^+$ and $u_{0,0,z}$, it follows that $$u_{0,0,z}u_{x,y,0}\Gamma\in E(F^+) \quad \Longleftrightarrow  \quad u_{x,y,0}\Gamma\in E(F^+) \quad \Longleftrightarrow  \quad (x,y)\in\Bad.$$
Hence, the HAW property of the set $\{u\in U : u\Gamma \in E(F^+)\}$ follows from the same property of $\Bad$ and Lemma \ref{L:HAW-mnfd} (v). Consequently, it suffices to prove Theorem \ref{T:HAW_U} in the case $\lambda > \mu$, which will be our standing assumption in \S\S3--5.
\end{remark}

For technical reasons, we will prove Theorem \ref{T:HAW_U} by applying Lemma \ref{L:HAW-mnfd} (iii) to the diffeomorphism $\RR^3\to U$,
$(x,y,z)\mt u_{x,y,z}^{-1}$. %Let $\lambda$, $\mu$ be as in Theorem \ref{T:HAW_U}.
We will need the following Diophantine characterization of the boundedness of $F^+_{\lambda,\mu}u_{x,y,z}^{-1}\Gamma$. For $\epsilon>0$ and $\bv=(p,r,q)\in\ZZ^2\times\NN$, we denote
$$\Delta_\epsilon(\bv):=\Big\{(x,y,z)\in\RR^3:\Big|x-\frac{p}{q}-z\Big(y-\frac{r}{q}\Big)\Big|<\frac{\epsilon}{q^{1+\lambda}}, \Big|y-\frac{r}{q}\Big|<\frac{\epsilon}{q^{1+\mu}}\Big\}.$$
Let
\begin{equation}\label{E:S-epsilon}
S_\epsilon(\lambda,\mu):=\RR^3\setminus\bigcup_{\bv\in\ZZ^2\times\NN}\Delta_\epsilon(\bv)
\end{equation}
and
\begin{equation}\label{E:S}
S(\lambda,\mu):=\bigcup_{\epsilon>0}S_\epsilon(\lambda,\mu).
\end{equation}

\begin{lemma}%[\cite{Kl}]
\label{L:diophantine}
The trajectory $F^+_{\lambda,\mu}u_{x,y,z}^{-1}\Gamma$ is bounded if and only if $(x,y,z)\in S(\lambda,\mu)$, that is, there is $\epsilon=\epsilon(x,y,z)>0$ such that
\begin{equation}\label{E:diophantine}
\max\left\{q^\lambda|qx-p-z(qy-r)|,q^\mu|qy-r|\right\}\ge\epsilon, \qquad \forall \ (p,r,q)\in\ZZ^2\times\NN.
\end{equation}
\end{lemma}

\begin{proof}
The lemma is a special case of \cite[Theorem 2.5]{Kl}. We repeat the proof for completeness. Let $g_t$ be as in  \equ{sldgtweights}.
By Mahler's compactness criterion, $F^+_{\lambda,\mu}u_{x,y,z}^{-1}\Gamma$ is bounded if and only if there is $\delta\in(0,1]$ such that
$\|g_tu_{x,y,z}^{-1}(p,r,q)^T\|_\infty\ge\delta$ for any $t\ge0$ and $(p,r,q)\in\ZZ^3\setminus\{\textbf{0}\}$,
where $\|\cdot\|_\infty$ is the sup-norm and the superscript ``$T$" denotes the transpose. A straightforward calculation shows that
$$g_tu_{x,y,z}^{-1}(p,r,q)^T=\left(e^{\lambda t}\big(p-qx-z(r-qy)\big),e^{\mu t}(r-qy),e^{-t}q\right)^T.$$
Thus $F^+_{\lambda,\mu}u_{x,y,z}^{-1}\Gamma$ is bounded if and only if there is $\delta\in(0,1]$ such that
\begin{equation}\label{E:diophantine1}
\max\left\{e^{\lambda t}|qx-p-z(qy-r)|,e^{\mu t}|qy-r|,e^{-t}|q|\right\}\ge\delta, \qquad \forall \ t\ge0, \ (p,r,q)\in\ZZ^3\setminus\{\textbf{0}\}.
\end{equation}

Suppose that there is $\delta\in(0,1]$ such that \eqref{E:diophantine1} holds. We claim that \eqref{E:diophantine} holds for $\epsilon=2^{-\lambda}\delta^{1+\lambda}$. In fact, for $(p,r,q)\in\ZZ^2\times\NN$, let $t_0\ge0$ be such that $e^{-t_0}q=\delta/2$. Then \eqref{E:diophantine1} implies that
\begin{align*}
& \ \max\left\{q^\lambda|qx-p-z(qy-r)|,q^\mu|qy-r|\right\}\\
\ge & \ 2^{-\lambda}\delta^\lambda\max\left\{e^{\lambda t_0}|qx-p-z(qy-r)|,e^{\mu t_0}|qy-r|\right\}\\
\ge & \ 2^{-\lambda}\delta^\lambda\cdot\delta=\epsilon.
\end{align*}
This verifies the claim.

Conversely, suppose that there is $\epsilon>0$ such that \eqref{E:diophantine} holds. We prove that \eqref{E:diophantine1} holds for $\delta=\min\{\epsilon^{\frac{1}{1+\lambda}},1\}$. Suppose not. Then there exist $(p,r,q)\in\ZZ^3\setminus\{\textbf{0}\}$ and $t\ge0$ such that
$$e^{\lambda t}|qx-p-z(qy-r)|<\delta, \qquad e^{\mu t}|qy-r|<\delta, \qquad e^{-t}|q|<\delta.$$
By replacing $(p,r,q)$ with $-(p,r,q)$ if necessary, we may assume that $q\ge0$.
If $q=0$, then $e^{\mu t}|r|<\delta$, which implies that $r=0$. In turn, we have $e^{\lambda t}|p|<\delta$, which implies that $p=0$. This is impossible. Thus $q>0$. It follows that
\begin{align*}
& \ \max\{q^\lambda|qx-p-z(qy-r)|,q^\mu|qy-r|\}\\
= & \ \max\{(e^{-t}q)^\lambda\cdot e^{\lambda t}|qx-p-z(qy-r)|,(e^{-t}q)^\mu\cdot e^{\mu t}|qy-r|\}\\
< & \ \max\{\delta^{1+\lambda},\delta^{1+\mu}\}=\delta^{1+\lambda}\le\epsilon.
\end{align*}
This contradicts \eqref{E:diophantine}. Thus, the proof of the lemma is complete.
\end{proof}

In view of Lemma \ref{L:diophantine} and Lemma \ref{L:HAW-mnfd} (iii), to prove Theorem \ref{T:HAW_U}, it suffices to prove that the set $S(\lambda,\mu)$ is HAW.
% on $\RR^3$.
We begin with the following simple lemma, which is a variant of \cite[Lemma 1]{BPV}.

\begin{lemma}\label{L:BPV}
Let $z\in\RR$. For any $(p,r,q)\in\ZZ^2\times\NN$, there exists $(a,b,c)\in\ZZ^3$ with $(a,b)\ne(0,0)$ such that $ap+br+cq=0$ and $|a|\le q^\lambda$, $|b+za|\le q^\mu$.
\end{lemma}

\begin{proof}
By Minkowski's linear forms theorem, there exist $a,b,c\in\ZZ$, not all zero, such that
$$|ap+br+cq|<1, \qquad |a|\le q^\lambda, \qquad |b+za|\le q^\mu.$$
Since $ap+br+cq\in\ZZ$, it must be $0$. If $a=b=0$, then it follows from $q\ne0$ that $c=0$, a contradiction. Thus $(a,b)\ne(0,0)$. This completes the proof.
\end{proof}

We now introduce some notation.
For a closed ball $B\subset\RR^3$ and $\bv=(p,r,q)\in\ZZ^2\times\NN$, consider the set of integral vectors
\begin{align*}
\cW(B,\bv):=\big\{(a,b,c)\in\ZZ^3: \ &  (a,b)\ne(0,0), ap+br+cq=0,\\
& |a|\le q^\lambda, |b+z_B a|\le q^\mu+\rho(B)^{\frac{1}{2}}\big\},
\end{align*}
where $z_B$ is the $z$-coordinate of the center of $B$. It follows from Lemma \ref{L:BPV} that $\cW(B,\bv)\ne\emptyset$ (the extra term $\rho(B)^{\frac{1}{2}}$ will be important for establishing Lemma \ref{L:main1} below). We choose and fix $$\bw(B,\bv)=\big(a(B,\bv),b(B,\bv),c(B,\bv)\big)\in\cW(B,\bv)$$ such that
\begin{align}
&\ \max\big\{|a(B,\bv)|,|b(B,\bv)+z_B a(B,\bv)|\big\} \notag\\
=&\ \min\left\{\max\big\{|a|,|b+z_B a|\big\}:(a,b,c)\in\cW(B,\bv)\right\}, \label{E:max}
\end{align}
and define
$$H_B(\bv):=q\max\big\{|a(B,\bv)|,|b(B,\bv)+z_B a(B,\bv)|\big\}.$$

\begin{lemma}\label{L:q}
For any closed ball $B\subset\RR^3$ and $\bv=(p,r,q)\in\ZZ^2\times\NN$, we have
\begin{equation}\label{E:qq}
q\le H_B(\bv)\le q^{1+\lambda}.
\end{equation}
\end{lemma}

\begin{proof}
The first inequality is obvious. By Lemma \ref{L:BPV}, $\cW(B,\bv)$ contains a vector $(a_0,b_0,c_0)$ with $|b_0+z_B a_0|\le q^\mu$. Thus, it follows from \eqref{E:max} that
$$\max\big\{|a(B,\bv)|,|b(B,\bv)+z_B a(B,\bv)|\big\}\le\max\big\{|a_0|,|b_0+z_B a_0|\big\}\le\max\{q^\lambda,q^\mu\}=q^\lambda.$$
Hence the second inequality.
\end{proof}

Now let $B_0\subset\RR^3$ be a closed ball of radius $\rho_0\le1$. Let $\kappa>1$ be such that
\begin{equation}\label{E:kappa}
\max\{|x|,|y|,|z|\}\le \kappa-1, \qquad \forall \ (x,y,z)\in B_0.
\end{equation}
Let $\beta\in(0,1)$, and $R$ and $\epsilon$ be positive numbers such that
\begin{equation}\label{E:R}
R\ge\max\{4\beta^{-1},10^7\kappa^4\}
\end{equation}
and
\begin{equation}\label{E:epsilon}
\epsilon\le 10^{-2}\kappa^{-2}R^{-10}\rho_0.
\end{equation}
Let $\sB_0=\{B_0\}$.
For $n\ge1$, let $\sB_n$ be the family of closed balls defined by
$$\sB_n:=\{B\subset B_0:\beta R^{-n}\rho_0<\rho(B)\le R^{-n}\rho_0\}.$$
For a closed ball $B$, if
\begin{equation}\label{E:condition0}
\text{$B\in\sB_n$ for some $n\ge0$,}
\end{equation}
we define
$$\cV_B:=\{\bv\in\ZZ^2\times\NN: H_n\le H_B(\bv)\le 2H_{n+1}\},$$
where
$$H_n=3\epsilon \kappa \rho_0^{-1}R^n, \qquad n\ge0.$$
Note that since $R>\beta^{-1}$, the families $\sB_n$ are pairwise disjoint. So \eqref{E:condition0} is satisfied for at most one integer $n$, and hence $\cV_B$ is well-defined.
It follows from \eqref{E:qq} that if $\bv=(p,r,q)\in\cV_B$, then
\begin{equation}\label{E:q}
H_n^{\frac{1}{1+\lambda}}\le q\le 2H_{n+1}.
\end{equation}
Whenever \eqref{E:condition0} is satisfied, we also define
$$\cV_{B,1}:=\{(p,r,q)\in\cV_B: H_n^{\frac{1}{1+\lambda}}\le q\le H_n^{\frac{1}{1+\lambda}}R^8\}$$
and
$$\cV_{B,k}:=\{(p,r,q)\in\cV_B: H_n^{\frac{1}{1+\lambda}}R^{2k+4}\le q\le H_n^{\frac{1}{1+\lambda}}R^{2k+6}\}, \qquad k\ge2.$$

\begin{lemma}\label{L:part}
If $B\in\sB_n$, then $\cV_B=\bigcup_{k=1}^n\cV_{B,k}$.
\end{lemma}

\begin{proof}
In view of \eqref{E:q}, it suffices to show that if $k>n$ then $\cV_{B,k}=\emptyset$. Suppose to the contrary that $\cV_{B,k}\ne\emptyset$ for some $k>n$. Let $(p,r,q)\in\cV_{B,k}$. Then
$$2H_{n+1}\ge q\ge H_n^{\frac{1}{1+\lambda}}R^{2k+4}> H_n^{\frac{1}{1+\lambda}}R^{2n+4}.$$
But
$$\frac{2H_{n+1}}{H_n^{\frac{1}{1+\lambda}}R^{2n+4}}=2(3\epsilon\kappa \rho_0^{-1})^{\frac{\lambda}{1+\lambda}}R^{-\frac{2+\lambda}{1+\lambda}n-3}<1,$$
a contradiction.
\end{proof}

Next, we inductively define a subfamily $\sB'_n$ of $\sB_n$ as follows. Let $\sB'_0=\{B_0\}$. If $n\ge1$ and $\sB'_{n-1}$ has been defined, we let
\begin{equation}\label{E:B'}
\sB'_n:=\Big\{B\in\sB_n:B\subset B' \text{ for some } B'\in\sB'_{n-1},  \text{ and } B\cap\bigcup_{\bv\in\cV_B}\Delta_\epsilon(\bv)=\emptyset\Big\}.
\end{equation}

The following two lemmas concerning $\sB'_n$ are key steps in the proof of Theorem \ref{T:HAW_U}. Their proofs are technical and postponed to the next two sections.

\begin{lemma}\label{L:main1}
Let $n\ge0$, $B\in\sB'_n$, $\bv=(p,r,q)\in\ZZ^2\times\NN$. If $q^{1+\lambda}\le 2H_{n+1}$, then $\Delta_\epsilon(\bv)\cap B=\emptyset$.
\end{lemma}

\begin{lemma}\label{L:main2}
Let $n\ge0$, $B\in\sB'_n$, and $k\ge 1$. There exists an affine plane $L_k(B)\subset\RR^3$ such that for any $B'\in\sB_{n+k}$ with $B'\subset B$ and any $\bv\in\cV_{B',k}$, we have that $$\Delta_\epsilon(\bv)\cap B'\subset L_k(B)^{(R^{-(n+k)}\rho_0)}.$$
\end{lemma}

We now prove Theorem \ref{T:HAW_U} assuming the truth of Lemmas \ref{L:main1} and \ref{L:main2}.

\begin{proof}[Proof of Theorem \ref{T:HAW_U}]
As remarked above, it suffices to prove that the set $S(\lambda,\mu)$ defined in \eqref{E:S} is HAW. In turn, by
Lemma \ref{L:HPW}, it suffices to prove that $S(\lambda,\mu)$ is HPW. Let $\beta\in(0,1)$, $\gamma>0$. We prove that Alice has a strategy to win the $(\beta,\gamma)$-hyperplane potential game on $\RR^3$ with target set $S(\lambda,\mu)$. In the first round of the game, Bob chooses a closed ball $B_0\subset\RR^3$. By Remark \ref{R:to0}, we may without loss of generality assume that Bob will play so that $\rho_i:=\rho(B_i)\to0$. By letting Alice make dummy moves in the first several rounds and relabeling $B_i$, we may also assume that $\rho_0\le1$. Let $\kappa$ and $R$ be positive numbers satisfying \eqref{E:kappa} and \eqref{E:R}. We also require that
\begin{equation}\label{E:R2}
(R^\gamma-1)^{-1}\le (\beta^2/2)^\gamma.
\end{equation}
Let $\epsilon, \sB_n, \cV_B, \cV_{B,k}$ and $\sB'_n$ be as above. Then Lemmas \ref{L:part}--\ref{L:main2} hold.
Let Alice play according to the following strategy: Suppose that $i\ge0$, and Bob has chosen the closed ball $B_i$. If there is $n\ge0$ such that
\begin{equation}\label{E:condition}
\text{$B_i\in\sB'_n$, and $i$ is the smallest nonnegative integer with $B_i\in\sB_n$,}
\end{equation}
then Alice chooses the family of neighborhoods $\{L_k(B_i)^{(2R^{-(n+k)}\rho_0)}:k\in\NN\}$, where $L_k(B_i)$ are the planes given in Lemma \ref{L:main2}. Otherwise, Alice makes an arbitrary move. Since $B_i\in\sB_n$, we have that $\rho_i>\beta R^{-n}\rho_0$. This, together with \eqref{E:R2}, implies that
$$\sum_{k=1}^\infty(2R^{-(n+k)}\rho_0)^\gamma=(2R^{-n}\rho_0)^\gamma(R^\gamma-1)^{-1}\le(\beta \rho_i)^\gamma.$$
So Alice's move is legal. We prove that this guarantees a win for Alice, that is, the unique point $\bx_\infty$ in $\bigcap_{i=0}^\infty B_i$ lies in $$S(\lambda,\mu)\cup\bigcup_{n\in\cN}\bigcup_{k=1}^\infty L_k(B_{i_n})^{(2R^{-(n+k)}\rho_0)},$$ where
$$\cN:=\{n\ge0:\text{ there exists } i=i_n \text{ such that \eqref{E:condition} holds}\}.$$ There are two cases.

\medskip

\textbf{Case (1).} Assume $\cN=\NN\cup\{0\}$. For any $\bv=(p,r,q)\in\ZZ^2\times\NN$, there is $n$ such that $q^{1+\lambda}\le 2H_{n+1}$. Since $n\in\cN$, we have that $B_{i_n}\in\sB'_n$. Then Lemma \ref{L:main1} implies that $\Delta_\epsilon(\bv)\cap B_{i_n}=\emptyset$. Thus $\bx_\infty\notin\Delta_\epsilon(\bv)$. It then follows from \eqref{E:S-epsilon} and the arbitrariness of $\bv$ that $\bx_\infty\in S_\epsilon(\lambda,\mu)\subset S(\lambda,\mu)$. Hence Alice wins.

\medskip

\textbf{Case (2).} Assume $\cN\ne\NN\cup\{0\}$. Let $n$ be the smallest nonnegative integer with $n\notin\cN$. Since $B_0\in\sB'_0$, we have $0\in\cN$. Hence $n\ge1$. Since $\rho_i\to0$ and $\rho_{i+1}\ge\beta \rho_i$, there must exist $i\ge1$ with $B_i\in\sB_n$ and $B_{i-1}\notin\sB_n$. It follows that $B_i\notin\sB'_n$. Since $n-1\in\cN$, we have $B_{i_{n-1}}\in\sB'_{n-1}$. Thus it follows from \eqref{E:B'} that  $B_i\cap\bigcup_{\bv\in\cV_{B_i}}\Delta_\epsilon(\bv)\ne\emptyset$. By Lemma \ref{L:part}, there is $k\in\{1,\ldots,n\}$ and $\bv\in\cV_{B_i,k}$ such that $B_i\cap\Delta_\epsilon(\bv)\ne\emptyset$. Applying Lemma \ref{L:main2} to $B=B_{i_{n-k}}$ and $B'=B_i$, we obtain
$B_i\cap\Delta_\epsilon(\bv)\subset L_k(B_{i_{n-k}})^{(R^{-n}\rho_0)}$.
So $B_i\cap L_k(B_{i_{n-k}})^{(R^{-n}\rho_0)}\ne\emptyset$. In view of $\rho_i\le R^{-n}\rho_0$, it follows that
$\bx_\infty\in B_i\subset L_k(B_{i_{n-k}})^{(2R^{-n}\rho_0)}$. Hence Alice also wins.

\medskip

This completes the proof of Theorem \ref{T:HAW_U} modulo Lemmas \ref{L:main1} and \ref{L:main2}.
\end{proof}

\section{Proof of Lemma \ref{L:main1}}

We now prove Lemma \ref{L:main1}.
Note that $2H_1=6\epsilon\kappa \rho_0^{-1}R<1$. So we may assume that $n\ge1$. We denote $B_n=B$, and let $B_n\subset\cdots\subset B_0$ be such that $B_k\in\sB'_k$. Suppose to the contrary that the conclusion of the lemma is not true. Then there exists $\bv=(p,r,q)\in\ZZ^2\times\NN$ with $q^{1+\lambda}\le 2H_{n+1}$ such that $\Delta_\epsilon(\bv)\cap B_k\ne\emptyset$ for every $1\le k\le n$. It follows from the definition of $\sB'_k$ that $\bv\notin\cV_{B_k}$, that is,
\begin{equation}\label{E:notin}
H_{B_k}(\bv)\notin[H_k,2H_{k+1}], \qquad 1\le k\le n.
\end{equation}
Let $1\le n_0\le n$ be such that
\begin{equation}\label{E:n_0}
2H_{n_0}\le q^{1+\lambda}\le 2H_{n_0+1}.
\end{equation}
We inductively prove that
\begin{equation}\label{E:induc}
H_{B_k}(\bv)<H_k, \qquad 1\le k\le n_0.
\end{equation}
Since $H_{B_{n_0}}(\bv)\le q^{1+\lambda}\le 2H_{n_0+1}$,
it follows from \eqref{E:notin} that
\eqref{E:induc} holds for $k=n_0$. Suppose that $1\le k\le n_0-1$ and \eqref{E:induc} holds if $k$ is replaced by $k+1$.
We prove that
\begin{equation}\label{E:conti}
H_{B_k}(\bv)\le2H_{B_{k+1}}(\bv).
\end{equation}
Denote $\bw(B_k,\bv)=(a_k,b_k,c_k)$, $z_{B_k}=z_k$.
Firstly, we notice that
\begin{align}
&\ |b_{k+1}+z_ka_{k+1}|\notag\\
\le&\ |b_{k+1}+z_{k+1}a_{k+1}|+|a_{k+1}||z_k-z_{k+1}|\notag\\
\le&\ |b_{k+1}+z_{k+1}a_{k+1}|+|a_{k+1}|\rho(B_k). \label{E:abz}
\end{align}
Secondly, we verify that
\begin{equation}\label{E:k+1}
\bw(B_{k+1},\bv)\in\cW(B_k,\bv).
\end{equation}
Since $\bw(B_{k+1},\bv)\in\cW(B_{k+1},\bv)$, we trivially have that
$$(a_{k+1},b_{k+1})\ne(0,0), \qquad a_{k+1}p+b_{k+1}r+c_{k+1}q=0, \qquad |a_{k+1}|\le q^\lambda.$$
On the other hand, it follows from \eqref{E:abz} that
\begin{align*}
&\ |b_{k+1}+z_ka_{k+1}|\\
\le&\ q^\mu+\rho(B_{k+1})^{\frac{1}{2}}+q^{-1}H_{B_{k+1}}(\bv)\rho(B_k)\\
\le&\ q^\mu+(\beta R)^{-\frac{1}{2}}\rho(B_k)^{\frac{1}{2}}+H_{n_0}^{-\frac{1}{2}}H_{k+1}\rho(B_k) \qquad \text{(by \eqref{E:n_0} and the induction hypothesis)}\\
\le&\ q^\mu+\frac{1}{2}\rho(B_k)^{\frac{1}{2}}+(3\epsilon\kappa R)^{\frac{1}{2}}\rho(B_k)^{\frac{1}{2}} \\
\le&\ q^\mu+\rho(B_k)^{\frac{1}{2}}.
\end{align*}
Hence \eqref{E:k+1} holds. It then follows from \eqref{E:abz}, \eqref{E:k+1} and the definition of $\bw(B_k,\bv)$ that
\begin{align*}
H_{B_k}(\bv)&=q\max\{|a_k|,|b_k+z_ka_k|\}\\
&\le q\max\{|a_{k+1}|,|b_{k+1}+z_ka_{k+1}|\}\\
&\le2q\max\{|a_{k+1}|,|b_{k+1}+z_{k+1}a_{k+1}|\}\\
&=2H_{B_{k+1}}(\bv).
\end{align*}
This proves \eqref{E:conti}.

It follows from \eqref{E:conti} and the induction hypothesis that $H_{B_k}(\bv)\le 2H_{k+1}$.
By \eqref{E:notin}, we have $H_{B_k}(\bv)<H_k$.
Thus \eqref{E:induc} is proved.

By letting $k=1$ in \eqref{E:induc}, we get $H_{B_1}(\bv)<H_1<1$. But $H_{B_1}(\bv)\ge q\ge1$, a contradiction. This completes the proof of Lemma \ref{L:main1}.

\section{Proof of Lemma \ref{L:main2}}\label{S:pf_main}

In this section, we prove Lemma \ref{L:main2}. We first establish the following estimates.

\begin{lemma}\label{L:inequ}
Let $n\ge0$, $B\in\sB'_n$, $k\ge1$, and $B_j\in\sB_{n+k}$ with $B_j\subset B$, $j=1,2$.
Let $\bv_j=(p_j,r_j,q_j)\in\cV_{B_j,k}$ be such that $\Delta_\epsilon(\bv_j)\cap B\ne\emptyset$, and $\bw_j=\bw(B_j,\bv_j)$. Then
\begin{align}
|\bv_1\cdot\bw_2| & \le 6\epsilon R^{d_k}+72\epsilon\kappa^2\frac{q_1}{q_2}R^{k+1},\label{E:v1w2}\\
|\bv_2\cdot\bw_1| & \le 6\epsilon R^{d_k}+72\epsilon\kappa^2\frac{q_2}{q_1}R^{k+1},\label{E:v2w1}
\end{align}
where $$d_k=\begin{cases}
  8, & k=1,\\
  2, & k\ge2.
\end{cases}$$
\end{lemma}

\begin{proof}
Let $(x_j,y_j,z_j)\in\Delta_\epsilon(\bv_j)\cap B$. Then
$$\Big|x_j-\frac{p_j}{q_j}-z_j\Big(y_j-\frac{r_j}{q_j}\Big)\Big|<\frac{\epsilon}{q_j^{1+\lambda}}, \qquad  \Big|y_j-\frac{r_j}{q_j}\Big|<\frac{\epsilon}{q_j^{1+\mu}}.$$
The latter inequality implies that
$$\Big|\frac{r_j}{q_j}\Big|\le|y_j|+\frac{\epsilon}{q_j^{1+\mu}}\le \kappa.$$
Thus
\begin{align*}
&\ \Big|\frac{p_1}{q_1}-\frac{p_2}{q_2}-z_{B_2}\Big(\frac{r_1}{q_1}-\frac{r_2}{q_2}\Big)\Big|\\
=&\ \Big|-\Big(x_1-\frac{p_1}{q_1}-z_1\Big(y_1-\frac{r_1}{q_1}\Big)\Big)+\Big(x_2-\frac{p_2}{q_2}-z_2\Big(y_2-\frac{r_2}{q_2}\Big)\Big)\\
& \qquad +(x_1-x_2)+\frac{r_1}{q_1}(z_1-z_{B_2})+\frac{r_2}{q_2}(z_{B_2}-z_2)-(y_1z_1-y_2z_2)\Big|\\
\le&\ \Big|x_1-\frac{p_1}{q_1}-z_1\Big(y_1-\frac{r_1}{q_1}\Big)\Big|+\Big|x_2-\frac{p_2}{q_2}-z_2\Big(y_2-\frac{r_2}{q_2}\Big)\Big|\\
& \qquad +|x_1-x_2|+\Big|\frac{r_1}{q_1}\Big||z_1-z_{B_2}|+\Big|\frac{r_2}{q_2}\Big||z_{B_2}-z_2|+|y_1||z_1-z_2|+|z_2|y_1-y_2|\\
\le&\ \frac{\epsilon}{q_1^{1+\lambda}}+\frac{\epsilon}{q_2^{1+\lambda}}+10\kappa \rho(B)
\end{align*}
and
$$
\Big|\frac{r_1}{q_1}-\frac{r_2}{q_2}\Big|=\Big|-\Big(y_1-\frac{r_1}{q_1}\Big)+\Big(y_2-\frac{r_2}{q_2}\Big)+(y_1-y_2)\Big|
\le\frac{\epsilon}{q_1^{1+\mu}}+\frac{\epsilon}{q_2^{1+\mu}}+2\rho(B).
$$
Let $\bw_j=(a_j,b_j,c_j)$. It follows that
\begin{align*}
q_1^{-1}|\bv_1\cdot\bw_2|=&\ |(q_1^{-1}\bv_1-q_2^{-1}\bv_2)\cdot\bw_2|\\
=&\ \Big|a_2\Big(\frac{p_1}{q_1}-\frac{p_2}{q_2}\Big)+b_2\Big(\frac{r_1}{q_1}-\frac{r_2}{q_2}\Big)\Big| \\
=&\ \Big|a_2\Big(\frac{p_1}{q_1}-\frac{p_2}{q_2}-z_{B_2}\Big(\frac{r_1}{q_1}-\frac{r_2}{q_2}\Big)\Big)
+(b_2+z_{B_2}a_2)\Big(\frac{r_1}{q_1}-\frac{r_2}{q_2}\Big)\Big| \\
\le&\ |a_2|\Big(\frac{\epsilon}{q_1^{1+\lambda}}+\frac{\epsilon}{q_2^{1+\lambda}}+10\kappa \rho(B)\Big)+|b_2+z_{B_2}a_2|\Big(\frac{\epsilon}{q_1^{1+\mu}}+\frac{\epsilon}{q_2^{1+\mu}}+2\rho(B)\Big) \\
\le&\ q_2^\lambda\Big(\frac{\epsilon}{q_1^{1+\lambda}}+\frac{\epsilon}{q_2^{1+\lambda}}\Big)+2q_2^\mu\Big(\frac{\epsilon}{q_1^{1+\mu}}+\frac{\epsilon}{q_2^{1+\mu}}\Big)
+12\kappa \rho(B)\max\big\{|a_2|,|b_2+z_{B_2}a_2|\big\}\\
\le&\ \epsilon q_1^{-1}\Big(\frac{q_2^\lambda}{q_1^\lambda}+\frac{q_1}{q_2}+2\frac{q_2^\mu}{q_1^\mu}+2\frac{q_1}{q_2}\Big)+12\kappa R^{-n}\rho_0\cdot q_2^{-1}H_{B_2}(\bv_2) \\
\le &\ 6\epsilon q_1^{-1}R^{d_k}+72\epsilon\kappa^2q_2^{-1}R^{k+1}.
\end{align*}
This proves \eqref{E:v1w2}. A similar argument verifies \eqref{E:v2w1}.
\end{proof}

For a closed ball $B\subset\RR^3$ and $\bv=(p,r,q)\in\ZZ^2\times\NN$, we consider the plane
$$L(B,\bv)=\{(x,y,z)\in\RR^3:a(B,\bv)x+b(B,\bv)y+c(B,\bv)=0\}.$$
The next lemma states that many pairs $(B,\bv)$ share the same plane $L(B,\bv)$.

\begin{lemma}\label{L:const}
Let $n\ge0$, $B\in\sB'_n$, $k\ge1$. Then either
\begin{itemize}
  \item[(i)] there is a plane $L_k(B)$ such that for any $B'\in\sB_{n+k}$ with $B'\subset B$, if $\bv\in\cV_{B',k}$ and $\Delta_\epsilon(\bv)\cap B\ne\emptyset$, then $L(B',\bv)=L_k(B)$, or
     \item[(ii)] $k=1$, and there is   $\bv_0=(p_0,r_0,q_0)\in\ZZ^2\times\NN$ with $H_{n+1}^{\frac{1}{1+\lambda}}\le q_0\le 2H_{n+2}$ such that for any $B'\in\sB_{n+1}$ with $B'\subset B$, if $\bv\in\cV_{B',1}$ and $\Delta_\epsilon(\bv)\cap B\ne\emptyset$, then $\bv=t\bv_0$ for some $t\ge1$.
\end{itemize}
\end{lemma}

\begin{proof}
We proceed by considering two separate cases.

\medskip

\textbf{Case (1).} Suppose $k=1$. We consider two subcases.

\medskip

\textbf{Case (1.1).} The linear span of the set
$$\cV:=\bigcup_{\substack{B'\in\sB_{n+1}\\B'\subset B}}\{\bv\in\cV_{B',1}:\Delta_\epsilon(\bv)\cap B\ne\emptyset\}$$
is of dimension $\ge2$. We prove that (i) holds. It suffices to prove that for $j=1,2$, if $B_j\in\sB_{n+1}$, $B_j\subset B$, $\bv_j\in\cV_{B_j,1}$, $\Delta_\epsilon(\bv_j)\cap B\ne\emptyset$, and $\bv_1$ and $\bv_2$  are linearly independent, then $L(B_1,\bv_1)=L(B_2,\bv_2)$. Let $\bv_j=(p_j,r_j,q_j)$, $\bw_j=\bw(B_j,\bv_j)$. It follows from Lemma \ref{L:inequ} that
$$|\bv_1\cdot\bw_2|\le 6\epsilon R^8+72\epsilon\kappa^2\frac{q_1}{q_2}R^2\le 78\epsilon\kappa^2R^{10}<1.$$
Since $\bv_1\cdot\bw_2$ is an integer, it must be $0$. Hence $\bw_2\, \bot\, \mathrm{span}\{\bv_1,\bv_2\}$. Similarly, $\bw_1\, \bot\, \mathrm{span}\{\bv_1,\bv_2\}$. Thus $\bw_1$ and $\bw_2$ are linearly dependent. This means that $L(B_1,\bv_1)=L(B_2,\bv_2)$.

\medskip

\textbf{Case (1.2).} The linear span of $\cV$ is of dimension $1$. We prove that (ii) holds. Let $\bv_0=(p_0,r_0,q_0)$ be an element in $\cV$ with the smallest possible $q_0$. Then any $\bv\in\cV$ is of the form $t\bv_0$ for some $t\ge1$. Moreover, since $\bv_0\in\cV_{B'}$ for some $B'\in\sB_{n+1}$, we have that $H_{n+1}^{\frac{1}{1+\lambda}}\le q_0\le 2H_{n+2}$. This completes the proof of Case (1).

\medskip

\textbf{Case (2).} Suppose $k\ge2$. We prove that (i) holds. It suffices to prove that if $B_j\in\sB_{n+k}$, $B_j\subset B$, $\bv_j\in\cV_{B_j,k}$, $\Delta_\epsilon(\bv_j)\cap B\ne\emptyset$, and $\bw_j=\bw(B_j,\bv_j)$  ($j=1,2$), then $\bw_1$ and $\bw_2$ are linearly dependent. Let $\bv_j=(p_j,r_j,q_j)$, $\bw_j=(a_j,b_j,c_j)$. Then

$$\max\{|a_j|,|b_j+z_{B_j}a_j|\}=q_j^{-1}H_{B_j}(\bv_j)\le(H_{n+k}^{\frac{1}{1+\lambda}}R^{2k+4})^{-1}\cdot 2H_{n+k+1}=2H_{n+k}^{\frac{\lambda}{1+\lambda}}R^{-2k-3}.$$
Moreover, for any $(x,y,z)\in B$, we have
\begin{align*}
|b_j+za_j|\le&\ |b_j+z_{B_j}a_j|+|a_j||z-z_{B_j}|\\
\le&\ q_j^\mu+\rho(B_j)^{\frac{1}{2}}+|a_j|\cdot2\rho(B)\\
\le&\ q_j^\mu+(R^{-(n+k)}\rho_0)^{\frac{1}{2}}+2H_{n+k}^{\frac{\lambda}{1+\lambda}}R^{-2k-3}\cdot 2R^{-n}\rho_0\\
=&\ q_j^\mu+(R^{-(n+k)}\rho_0)^{\frac{1}{2}}+
4(3\epsilon\kappa)^{\frac{\lambda}{1+\lambda}}\rho_0^{\frac{1}{1+\lambda}}R^{-(3+\frac{1}{1+\lambda}n+\frac{2+\lambda}{1+\lambda}k)}\\
\le&\ q_j^\mu+R^{-\frac{1}{2}}+R^{-1} \\
\le&\ q_j^\mu+1\le2q_j^\mu\le 2(H_{n+k}^{\frac{1}{1+\lambda}}R^{2k+6})^\mu\le2H_{n+k}^{\frac{\mu}{1+\lambda}}R^{k+3}.
\end{align*}
Furthermore, it follows from Lemma \ref{L:inequ} that
\begin{align*}
q_1^{-1}|\bv_1\cdot\bw_2|&\le6\epsilon q_1^{-1}R^2+72\epsilon\kappa^2q_2^{-1}R^{k+1}\\
&\le78\epsilon\kappa^2\cdot(H_{n+k}^{\frac{1}{1+\lambda}}R^{2k+4})^{-1}\cdot R^{k+1}\\
&=78\epsilon\kappa^2H_{n+k}^{-\frac{1}{1+\lambda}}R^{-k-3}.
\end{align*}
Let $$\bv_0=(p_0,r_0,q_0)=\bw_1\times \bw_2.$$
It suffices to prove that $\bv_0=0$. By the triple cross product expansion, we have
$$\bv_1\times \bv_0=\bv_1\times(\bw_1\times \bw_2)=(\bv_1\cdot\bw_2)\bw_1.$$
Comparing the first two components of the vectors on both sides, we obtain
\begin{align}
q_0\frac{r_1}{q_1}-r_0&=q_1^{-1}(\bv_1\cdot\bw_2)a_1, \label{E:c1}\\
q_0\frac{p_1}{q_1}-p_0&=-q_1^{-1}(\bv_1\cdot\bw_2)b_1. \label{E:c2}
\end{align}
Suppose to the contrary that $\bv_0\ne0$. There are two cases.

\medskip

\textbf{Case (2.1).} Suppose $q_0=0$. Then $(p_0,r_0)\ne(0,0)$. It then follows from \eqref{E:c1} and \eqref{E:c2} that
\begin{align*}
1&\le\max\{|r_0|,|p_0-z_{B_1}r_0|\}\\
&=q_1^{-1}|\bv_1\cdot\bw_2|\max\{|a_1|,|b_1+z_{B_1}a_1|\}\\
&\le q_1^{-1}|\bv_1\cdot\bw_2|\max\{|a_1|,|b_1+z_{B_1}a_1|\}^{1/\lambda}\\
&\le 78\epsilon\kappa^2H_{n+k}^{-\frac{1}{1+\lambda}}R^{-k-3}\cdot 4H_{n+k}^{\frac{1}{1+\lambda}}R^{-(2k+3)/\lambda}\\
&\le312\epsilon\kappa^2<1,
\end{align*}
a contradiction.

\medskip

\textbf{Case (2.2).} Suppose $q_0\ne0$. Without loss of generality, we may assume that $q_0>0$. Then
\begin{align}
q_0&=|a_1b_2-a_2b_1|=|a_1(b_2+za_2)-a_2(b_1+za_1)|\notag\\
&\le 2\cdot2H_{n+k}^{\frac{\lambda}{1+\lambda}}R^{-2k-3}\cdot2H_{n+k}^{\frac{\mu}{1+\lambda}}R^{k+3}=8H_{n+k}^{\frac{1}{1+\lambda}}R^{-k}.\label{E:q_0}
\end{align}
It follows that
$$q_0/q_1\le8H_{n+k}^{\frac{1}{1+\lambda}}R^{-k}\cdot H_{n+k}^{-\frac{1}{1+\lambda}}=8R^{-k}\le1/2.$$
We prove that $\Delta_\epsilon(\bv_1)\cap B\subset\Delta_\epsilon(\bv_0)$.
Let $(x,y,z)\in\Delta_\epsilon(\bv_1)\cap B$.
It follows from \eqref{E:c1} and \eqref{E:c2}  that
\begin{align*}
&\ q_0^{1+\lambda}\Big|x-\frac{p_0}{q_0}-z\Big(y-\frac{r_0}{q_0}\Big)\Big|\\
\le &\ q_0^{1+\lambda}\Big|x-\frac{p_1}{q_1}-z\Big(y-\frac{r_1}{q_1}\Big)\Big|
+q_0^{1+\lambda}\Big|\frac{p_1}{q_1}-\frac{p_0}{q_0}-z\Big(\frac{r_1}{q_1}-\frac{r_0}{q_0}\Big)\Big|\\
< &\ q_0^{1+\lambda}\frac{\epsilon}{q_1^{1+\lambda}}+q_0^\lambda q_1^{-1}|\bv_1\cdot\bw_2||b_1+za_1|\\
\le &\ \epsilon/2+8H_{n+k}^{\frac{\lambda}{1+\lambda}}R^{-\lambda k}\cdot78\epsilon\kappa^2H_{n+k}^{-\frac{1}{1+\lambda}}R^{-k-3}\cdot2H_{n+k}^{\frac{\mu}{1+\lambda}}R^{k+3}\\
= &\  \epsilon/2+1248\epsilon\kappa^2 R^{-\lambda k}\\
\le &\ \epsilon(1/2+1248\kappa^2 R^{-1/2})<\epsilon
\end{align*}
and
\begin{align*}
q_0^{1+\mu}\Big|y-\frac{r_0}{q_0}\Big|
\le &\ q_0^{1+\mu}\Big|y-\frac{r_1}{q_1}\Big|+q_0^{1+\mu}\Big|\frac{r_1}{q_1}-\frac{r_0}{q_0}\Big|\\
< &\ q_0^{1+\mu}\frac{\epsilon}{q_1^{1+\mu}}+q_0^\mu q_1^{-1}|\bv_1\cdot\bw_2||a_1|\\
\le &\ \epsilon/2+8H_{n+k}^{\frac{\mu}{1+\lambda}}R^{-\mu k}\cdot78\epsilon\kappa^2H_{n+k}^{-\frac{1}{1+\lambda}}R^{-k-3}\cdot2H_{n+k}^{\frac{\lambda}{1+\lambda}}R^{-2k-3}\\
= &\ \epsilon/2+1248\epsilon\kappa^2R^{-\mu k-3k-6}\\
\le &\ \epsilon(1/2+1248\kappa^2 R^{-1})<\epsilon.
\end{align*}
Thus $(x,y,z)\in\Delta_\epsilon(\bv_0)$. This proves $\Delta_\epsilon(\bv_1)\cap B\subset\Delta_\epsilon(\bv_0)$.
It then follows from $\Delta_\epsilon(\bv_1)\cap B\ne\emptyset$ that $\Delta_\epsilon(\bv_0)\cap B\ne\emptyset$. By Lemma \ref{L:main1}, we have
$q_0^{1+\lambda}>2H_{n+1}$, which contradicts \eqref{E:q_0}. This completes the proof of the lemma.
\end{proof}

We are now prepared to prove Lemma \ref{L:main2}.

\begin{proof}[Proof of Lemma \ref{L:main2}]
Let $n\ge0$, $B\in\sB'_n$, $k\ge 1$. We need to prove that there is a plane $L_k(B)$ such that for any $B'\in\sB_{n+k}$ with $B'\subset B$ and any $\bv\in\cV_{B',k}$, we have $\Delta_\epsilon(\bv)\cap B'\subset L_k(B)^{(R^{-(n+k)}\rho_0)}$. We need only to consider the case that $\Delta_\epsilon(\bv)\cap B'\ne\emptyset$.
By Lemma \ref{L:const}, one of the following statements holds:
\begin{itemize}
  \item[(i)] There is a plane $L_k(B)$ such that for any $B'\in\sB_{n+k}$ with $B'\subset B$, if $\bv\in\cV_{B',k}$ and $\Delta_\epsilon(\bv)\cap B\ne\emptyset$, then $L(B',\bv)=L_k(B)$.
  \item[(ii)] $k=1$, and there is $\bv_0=(p_0,r_0,q_0)\in\ZZ^2\times\NN$ with $H_{n+1}^{\frac{1}{1+\lambda}}\le q_0\le 2H_{n+2}$ such that for any $B'\in\sB_{n+1}$ with $B'\subset B$, if $\bv\in\cV_{B',1}$ and $\Delta_\epsilon(\bv)\cap B\ne\emptyset$, then $\bv=t\bv_0$ for some $t\ge1$.
\end{itemize}

Suppose (i) holds. It suffices to prove that $\Delta_\epsilon(\bv)\cap B'\subset L(B',\bv)^{(R^{-(n+k)}\rho_0)}$ for any $B'\in\sB_{n+k}$ and $\bv=(p,r,q)\in\cV_{B'}$.
Firstly, we have that
$$\rho(B')q\le R^{-(n+k)}\rho_0\cdot2H_{n+k+1}=6\epsilon \kappa R\le 1/4.$$
Denote $\bw(B',\bv)=(a,b,c)$. If $(x,y,z)\in\Delta_\epsilon(\bv)\cap B'$, then
\begin{align*}
  |ax+by+c|&=\Big|a\Big(x-\frac{p}{q}-z\Big(y-\frac{r}{q}\Big)\Big)+(b+za)\Big(y-\frac{r}{q}\Big)\Big|\\
  &< |a|\frac{\epsilon}{q^{1+\lambda}}+|b+za|\frac{\epsilon}{q^{1+\mu}}\\
  &\le \epsilon q^{-1}(|a|q^{-\lambda}+|b+z_{B'} a|q^{-\mu}+|z-z_{B'}||a|q^{-\mu})\\
  &\le \epsilon q^{-1}(2+\rho(B')^{\frac{1}{2}}q^{-\mu}+\rho(B')q^{\lambda-\mu})\\
  &\le \epsilon q^{-1}(2+(\rho(B')q)^{\frac{1}{2}}+\rho(B')q)\\
  &\le 3\epsilon q^{-1}.
\end{align*}
Note that
$$H_{B'}(\bv)=q\max\{|a|,|b+z_{B'}a|\}\le\kappa q\max\{|a|,|b|\}.$$
Thus
$$\frac{|ax+by+c|}{\sqrt{a^2+b^2}}<\frac{3\epsilon}{q\max\{|a|,|b|\}}\le\frac{3\epsilon \kappa }{H_{B'}(\bv)}\le\frac{3\epsilon \kappa }{H_{n+k}}=R^{-(n+k)}\rho_0.$$
This means that $(x,y,z)\in L(B',\bv)^{(R^{-(n+k)}\rho_0)}$. Thus $\Delta_\epsilon(\bv)\cap B'\subset L(B',\bv)^{(R^{-(n+k)}\rho_0)}$.

\medskip

Suppose (ii) holds. Consider the plane
$$L_1(B)=\Big\{(x,y,z)\in\RR^3:x-\frac{p_0}{q_0}-z_B\Big(y-\frac{r_0}{q_0}\Big)=0\Big\}.$$
We proceed by showing that $\Delta_\epsilon(\bv)\cap B\subset L_1(B)^{(R^{-(n+1)}\rho_0)}$ for any $\bv=t\bv_0$ with $t\ge1$. Let $\bv=(p,r,q)$. Then $q\ge q_0$, $p/q=p_0/q_0$, and $r/q=r_0/q_0$. If $(x,y,z)\in\Delta_\epsilon(\bv)\cap B$, then
\begin{align*}
&\ (1+z_B^2)^{-1/2}\Big|x-\frac{p_0}{q_0}-z_B\Big(y-\frac{r_0}{q_0}\Big)\Big|\\
\le&\ \Big|x-\frac{p}{q}-z\Big(y-\frac{r}{q}\Big)\Big|+|z-z_B|\Big|y-\frac{r}{q}\Big|\\
<&\ \frac{\epsilon}{q^{1+\lambda}}+\rho(B)\frac{\epsilon}{q^{1+\mu}}\\
\le&\ \frac{\epsilon}{q_0^{1+\lambda}}+\rho(B)\frac{\epsilon}{q_0^{1+\mu}}\\
\le&\ \frac{\epsilon}{q_0^{1+\lambda}}(1+\rho(B)q_0)\\
\le&\ \frac{\epsilon}{H_{n+1}}(1+R^{-n}\rho_0\cdot 2H_{n+2})\\
=&\ (3\kappa)^{-1}R^{-(n+1)}\rho_0(1+6\kappa\epsilon R^2)\\
<&\ R^{-(n+1)}\rho_0.
\end{align*}
This means that $(x,y,z)\in L_1(B)^{(R^{-(n+1)}\rho_0)}$. Thus $\Delta_\epsilon(\bv)\cap B\subset L_1(B)^{(R^{-(n+1)}\rho_0)}$.
\end{proof}

\begin{remark}
We say that a plane $L\subset\RR^3$ is \textsl{vertical} if it is of the form
$$L=\{(x,y,z)\in\RR^3:ax+by+c=0\},$$
where $a,b,c\in\RR$ and $(a,b)\ne(0,0)$. The above proof in fact shows that the planes $L_k(B)$ given in Lemma \ref{L:main2} are vertical. In turn, in view of \cite[Theorem C.8]{FSU} (the case of $Z=\RR^3$ and $\mathcal{H}=\{\text{vertical planes}\}$), the proof of Theorem \ref{T:HAW_U} actually shows that Alice can win the hyperplane absolute game on $\RR^3$ with target set $S(\lambda,\mu)$ by choosing neighborhoods of vertical planes.
%This implies that $S(\lambda,\mu)$ is HAW on every ``horizontal" plane. More precisely,
Furthermore, for any fixed $z\in\RR$ let us  denote
$$S(z;\lambda,\mu):=\{(x,y)\in\RR^2:(x,y,z)\in S(\lambda,\mu)\},$$
that is, the set of those $(x,y)$ for which there exists $c>0$ such that
$$
\max\left\{q^{1+\lambda}\Big|x-\frac{p}{q}-z\Big(y-\frac{r}{q}\Big)\Big|,q^{1+\mu}\Big|y-\frac{r}{q}\Big|\right\}\ge c, \qquad \forall \ (p,r,q)\in\ZZ^2\times\NN.
$$
It is easy to see that the vertical HAW property of $S(\lambda,\mu)$ yields the following: \textsl{For every $z\in\RR$, the set $S(z;\lambda,\mu)$ is HAW on $\RR^2$.} Note that $S(0;\lambda,\mu)$ coincides with $\Bad(\lambda,\mu)$; more generally, it is easy to see that if $z\in\QQ$, then $(x,y)\in S(z;\lambda,\mu)$ if and only if $(x-zy,y)\in\Bad(\lambda,\mu)$. However, if $z$ is irrational, we do not know any relation between $S(z;\lambda,\mu)$ and $\Bad(\lambda,\mu)$.
\end{remark}

\section{Proofs of Theorems \ref{T:HAW} and \ref{T:HAW-expand}}\label{S:6}

We are now in position to prove Theorems \ref{T:HAW} and \ref{T:HAW-expand}.

\begin{proof}[Proof of Theorem \ref{T:HAW}]
We first prove the theorem for $F^+=F^+_{\lambda,\mu}$ as in \equ{sldgtweights}, where $\lambda,\mu$ are as in \equ{condr} with $\lambda \ge \mu$.
In view of Lemma \ref{L:HAW-mnfd} (iv), it suffices to prove that for any $\Lambda\in X=G/\Gamma$, there is an open neighborhood $\Omega$ of $\Lambda$ in $X$ such that $E(F^+_{\lambda,\mu})\cap\Omega$ is HAW on $\Omega$.  Let $U$ be the group as in  \equ{defu}, and let $B$ be the group of lower triangular matrices in $G$. By the Bruhat decomposition, the set $BU$ is Zariski open in $G$, and the multiplication map $B\times U\to BU$ is a diffeomorphism.

We claim that there exist $b_0\in B$ and $u_0\in U$ such that $b_0u_0\Gamma=\Lambda$. In fact, by the Borel density theorem, the set $\pi^{-1}(\Lambda)$ is Zariski dense in $G$, where $\pi:G\to X$ is the natural projection. It follows that $\pi^{-1}(\Lambda)\cap BU\ne\emptyset$. Thus we can choose $b_0\in B$ and $u_0\in U$ such that $b_0u_0\in\pi^{-1}(\Lambda)$, that is, $b_0u_0\Gamma=\Lambda$. This verifies the claim.

Let $\Omega_B$ and $\Omega_U$ be open neighborhoods of $b_0$ and $u_0$ in $B$ and $U$, respectively, such that the map $\varphi:\Omega_B\times\Omega_U\to X$, $\varphi(b,u)=bu\Gamma$, is a diffeomorphism onto an open set $\Omega$ in $X$. It follows that $\Omega$ is an open neighborhood of $\Lambda$. Thus, in view of Lemma \ref{L:HAW-mnfd} (iii), it suffices to prove that the set
\begin{equation}\label{E:set}
\varphi^{-1}(E(F^+_{\lambda,\mu})\cap\Omega)=\left\{(b,u)\in\Omega_B\times\Omega_U:bu\Gamma\in E(F^+_{\lambda,\mu})\right\}
\end{equation}
is HAW on $\Omega_B\times\Omega_U$. Note that for $b\in B$, the subset $\{g_tbg_t^{-1}:t\ge0\}$ of $G$ is bounded. It follows that $bu\Gamma\in E(F^+_{\lambda,\mu})$ if and only if $u\Gamma\in E(F^+_{\lambda,\mu})$. This implies that the set \eqref{E:set} is equal to
\begin{equation}\label{E:set2}
\Omega_B\times\left\{u\in\Omega_U:u\Gamma\in E(F^+_{\lambda,\mu})\right\}.
\end{equation}
By Theorem \ref{T:HAW_U} and Lemma \ref{L:HAW-mnfd} (v), the set \eqref{E:set2} is HAW on $\Omega_B\times\Omega_U$. This proves that $E(F^+_{\lambda,\mu})$ is HAW on $X$.

\medskip

We now consider the case when $F^+$ is $\RR$-diagonalizable. In this case, it is easy to see that either $F^+$ is trivial (in which case the conclusion obviously holds), or there is an automorphism $\sigma$ of $G$ such that $F^+=\sigma(F^+_{\lambda,\mu})$ for some $\lambda,\mu$ as in \equ{condr} with $\lambda \ge \mu$. The automorphism $\sigma$ is of the form $\sigma(g)=g_0\tau(g)g_0^{-1}$, where $g_0\in G$ and $\tau(g)=g$ or $(g^T)^{-1}$. Note that $\tau$ preserves $\Gamma$, hence induces a diffeomorphism $\bar{\tau}$ of $X$ by $\bar{\tau}(g\Gamma)=\tau(g)\Gamma$. We claim that
\begin{equation}\label{E:set3}
E(F^+)=g_0\bar{\tau}\big(E(F^+_{\lambda,\mu})\big).
\end{equation}
In fact, for $\Lambda\in X$, we have
\begin{align*}
\Lambda\in E(F^+) \Longleftrightarrow & \ g_0^{-1}F^+\Lambda=\tau(F^+_{\lambda,\mu})g_0^{-1}\Lambda \text{ is bounded}\\
\Longleftrightarrow & \ \bar{\tau}\big(\tau(F^+_{\lambda,\mu})g_0^{-1}\Lambda\big)=F^+_{\lambda,\mu}\bar{\tau}(g_0^{-1}\Lambda) \text{ is bounded}\\
\Longleftrightarrow & \ \bar{\tau}(g_0^{-1}\Lambda)\in E(F^+_{\lambda,\mu})\\
\Longleftrightarrow & \ \Lambda\in g_0\bar{\tau}\big(E(F^+_{\lambda,\mu})\big).
\end{align*}
This verifies \eqref{E:set3}. The HAW property of $E(F^+)$ now follows from that of $E(F^+_{\lambda,\mu})$ and Lemma \ref{L:HAW-mnfd} (iii).

\medskip

Finally, we consider the general case when $F^+$ is only assumed to be diagonalizable. By the real Jordan decomposition, there are one-parameter subgroups $F_i=\{g_t^{(i)}:t\in\RR\}$ $(i=1,2)$ such that $F_1$ is $\RR$-diagonalizable, $F_2$ has compact closure, and $g_t=g_t^{(1)}g_t^{(2)}$ with $g_t^{(1)}$ commuting with $g_t^{(2)}$. It is easy to see that $E(F^+)=E(F^+_1)$. Thus the HAW property of $E(F^+)$ follows from that of $E(F^+_1)$. This completes the proof.
\end{proof}

Next, we prove Theorem \ref{T:HAW-expand}.

\begin{proof}[Proof of Theorem \ref{T:HAW-expand}]
We first consider the case when $F^+=F^+_{\lambda,\mu}$ with $\lambda\ge\mu$. There are two sub-cases.

\medskip

\textbf{Case (1).} Assume $\lambda>\mu$. Then $H$ is equal to the group $U$ as in \equ{defu}.  We need to prove that for any $\Lambda\in X$, the set of $u\in U$ such that $u\Lambda\in E(F^+_{\lambda,\mu})$ is HAW on $U$.
The special case of $\Lambda=\Gamma$ reduces to Theorem \ref{T:HAW_U}. We prove the general case by using this special case. In view of Lemma \ref{L:HAW-mnfd} (iv), it suffices to prove that for any $u_0\in U$, there is an open neighborhood $\Omega$ of $u_0$ in $U$ such that the set
\begin{equation}\label{E:HAW-Omega}
\{u\in\Omega:u\Lambda\in E(F^+_{\lambda,\mu})\}
\end{equation}
is HAW on $\Omega$. Similar to the proof of Theorem \ref{T:HAW}, the Bruhat decomposition and the Borel density theorem imply that $\pi^{-1}(\Lambda)\cap u_0^{-1}BU\ne\emptyset$, where $\pi:G\to G/\Gamma$ is the projection and $B$ is the group of lower triangular matrices in $G$.
Choose $g_0\in\pi^{-1}(\Lambda)\cap u_0^{-1}BU$. Then $\Lambda=g_0\Gamma$ and $u_0g_0\in BU$. Let $\Omega$ be an open neighborhood of $u_0$ in $U$ such that $\Omega g_0\subset BU$. Then there are smooth maps $\varphi:\Omega\to B$ and $\psi:\Omega\to U$ such that
\begin{equation}\label{E:Brubat}
ug_0=\varphi(u)\psi(u), \qquad \forall \ u\in\Omega.
\end{equation}
We verify that the tangent map $(d\psi)_{u_0}$ is a linear isomorphism. In fact, if we let $b_0=\varphi(u_0)$ and $u'_0=\psi(u_0)$,
then it follows from \eqref{E:Brubat} that
$$dr_{g_0}(Y)=dr_{u'_0}\circ(d\varphi)_{u_0}(Y)+dl_{b_0}\circ(d\psi)_{u_0}(Y), \qquad \forall \ Y\in T_{u_0}U,$$
where for $g\in G$, $r_g$ and $l_g$ denote the corresponding right and left translations on $G$, respectively.
Note that $$dr_{g_0}(Y)\in T_{u_0g_0}(Uu_0g_0), \qquad dr_{u'_0}\circ(d\varphi)_{u_0}(Y)\in T_{u_0g_0}(Bu_0g_0).$$
Thus if $(d\psi)_{u_0}(Y)=0$, then $$dr_{g_0}(Y)\in T_{u_0g_0}(Uu_0g_0)\cap T_{u_0g_0}(Bu_0g_0)=0,$$ and hence $Y=0$. This shows that $(d\psi)_{u_0}$ is an isomorphism.
In view of the inverse function theorem, by shrinking $\Omega$ if necessary, we may assume that $\psi$ is a diffeomorphism from $\Omega$ onto an open subset $\psi(\Omega)$ of $U$. Note that for $u\in\Omega$,
$$F^+_{\lambda,\mu}u\Lambda=F^+_{\lambda,\mu}ug_0\Gamma=F^+_{\lambda,\mu}\varphi(u)\psi(u)\Gamma$$ is bounded if and only if $F^+_{\lambda,\mu}\psi(u)\Gamma$ is bounded. Thus the image under $\psi$ of the set \eqref{E:HAW-Omega} is equal to
$$\{u'\in\psi(\Omega):u'\Gamma\in E(F^+_{\lambda,\mu})\},$$
which, by Theorem \ref{T:HAW_U}, is HAW on $\psi(\Omega)$. Hence, it follows from Lemma \ref{L:HAW-mnfd} (iii) that the set \eqref{E:HAW-Omega} is HAW on $\Omega$.

\medskip

\textbf{Case (2).} Assume $\lambda=\mu=1/2$. Then $H$ is equal to the group $U_0$ defined in \equ{defu0}. Let $\Lambda\in X$. We need to prove that the set of $u\in U_0$ such that $u\Lambda\in E(F^+_{\frac{1}{2},\frac{1}{2}})$ is HAW on $U_0$. The case of $\Lambda=\Gamma$ follows immediately from the HAW property of $\Bad$ proved in \cite{BFKRW}. For the general case, it suffices to prove that for any $u_0\in U_0$, there is an open neighborhood $\Omega$ of $u_0$ in $U_0$ such that the set
\begin{equation}\label{E:HAW-Omega2}
\{u\in\Omega:u\Lambda\in E(F^+_{\frac{1}{2},\frac{1}{2}})\}
\end{equation}
is HAW on $\Omega$. This can be done by modifying the proof of Case (1) as follows.
Choose $g_0\in G$ such that $\Lambda=g_0\Gamma$ and $u_0g_0\in BU$, and let $\Omega$ be an open neighborhood of $u_0$ in $U_0$ such that $\Omega g_0\subset BU$. Consider the group $$P:=\left\{\begin{pmatrix}
  *&*&0\\ *&*&0\\ *&*&*
\end{pmatrix}\in G\right\}.$$ There are smooth maps $\varphi:\Omega\to P$ and $\psi:\Omega\to U_0$ such that
$ug_0=\varphi(u)\psi(u)$ for any $u\in\Omega$.
(We can first decompose $ug_0$ into a product of elements in $B$ and $U$, and then decompose the $U$-component into a product of elements in $P$ and $U_0$.)
Similar to Case (1), we can show that, by shrinking $\Omega$ if necessary, $\psi$ is a diffeomorphism from $\Omega$ onto an open subset $\psi(\Omega)$ of $U_0$. In turn, the HAW property on $\psi(\Omega)$ of the set
$$\{u'\in\psi(\Omega):u'\Gamma\in E(F^+_{\frac{1}{2},\frac{1}{2}})\}$$ implies that the set \eqref{E:HAW-Omega2}
is HAW on $\Omega$.

\medskip

This completes the proof of Theorem \ref{T:HAW-expand} for the case $F^+=F^+_{\lambda,\mu}$ with $\lambda\ge\mu$.
The proofs of the more general cases are similar to the corresponding parts in the proof of Theorem~\ref{T:HAW}.
\end{proof}

\section{Generalizations}\label{S:7}

Let $G$ be a Lie group,  $\Gamma \subset G$ a non-uniform lattice, and let $F^+ $ be a one-parameter subsemigroup  of $G$.
Clearly the set $E(F^+)$ has zero Haar measure whenever the $F^+$-action on $\ggm$ is ergodic.
As mentioned in the introduction, Margulis \cite{Ma} conjectured   $E(F^+)$ to be thick whenever $F^+$ is non-quasiunipotent.
In \cite{KM} a necessary and sufficient condition for the thickness of  $E(F^+)$ was found -- so-called condition (Q); in particular, it is satisfied whenever
 %$G$ is semi-simple, $\Gamma$ is irreducible and
$F^+$ is $\Ad$-diagonalizable.

The aforementioned conjecture had been motivated by earlier results on the subject \cite{Da1, Da2}.
The latter were based on the method of Schmidt games.  However the argument in \cite{KM} did  not rely on games,
and  it was not possible to derive from it any information on the intersection of sets $E(F^+)$ for different $F^+$.
Later another proof of the thickness of sets $E(F^+)$  was found \cite{KW1}, based on the notion of modified Schmidt games.
Yet it also fails to be strong enough to guarantee that  in general,
$E(F^+_1)\cap E(F_2^+) \ne \emptyset$ for two different
$\Ad$-diagonalizable % non-quasiunipotent
semigroups $F_1^+, F_2^+ \subset G$.

Due to the approach pioneered in \cite{BPV} and extended in \cite{An1, An2, ABV, NS} and the present paper, we now know much more for $G$ and $\Gamma$ as in  \equ{sld}. Another very recent development comes from the paper \cite{Be} by Beresnevich.
Namely,
let us take \eq{sldgeneral}{G=\SL_{d+1}(\RR),\ \Gamma=\SL_{d+1}(\ZZ),\ X = \ggm\,,} and denote by $S_d$ the $(d-1)$-dimensional simplex
$$S_d = \left\{\br = (r_1,\dots,r_d): r_i \ge 0,\ \  \sum_{i=1}^d r_i = 1\right\}\,.$$
For $\br\in S_d$,  consider
 \eq{sldplusgt}{F_{\br}^+ = \{g_t: t \ge 0\}\text{ where }g_t = \diag(e^{r_1 t},\dots,e^{r_d t} ,e^{-t} )\in G\,.}
Also for $\x\in\RR^d$ denote $u_{\x} =\begin{pmatrix}
  I_d & \x\\  &1
\end{pmatrix}\in G$.
 %\eq{defu0plus}{U_0 = \{u_{\bf x} : \x\in\RR^d\}, \quad \text{where } u_{\x} =\begin{pmatrix}
%  I_d & \x\\  &1
%\end{pmatrix}\in G.}
Then it is shown in \cite{Kl} that $ u_{\x}\Gamma \in E( F^+_{\br} )$ if and only if $\x$ belongs to the set $$\Bad(\rr) :=\left\{\x\in\RR^d: \ \ %\exists\,c > 0\text{ such that }
\inf_{(p_1,\dots,p_d) \in \mathbb{Z}^{d},\ q\in
\mathbb{N}}\ \  \max_i\{ q^{r_i}|qx_i-p_i|\}> 0\right\}\,.$$
%{\sl badly approximable  with  weight vector} $\br$), that is, there is $c=c(\x)>0$ such that $$$.
This allows one to dynamically restate  \cite[Theorem 1]{Be} (not in its strongest form): {\it if $\mathcal{R}$ is a countable subset of $S_d$ with $\dist(\mathcal{R},\partial S_d) > 0$, then the set
$$
\bigcap_{\br\in\mathcal{R}}\left\{\x\in\RR^d: u_{\x}\Gamma \in E( F^+_{\br} )\right\}
$$
is thick. }

Note that the argument  in \cite{Be} does not involve Schmidt games, and it is not clear how to utilize it to conclude that the set
$\bigcap_{\br\in\mathcal{R}}\ E( F^+_{\br} )$ is thick in $X$. Still we would like to propose the following %bold
%We will denote the set of $(x,y)$ with this property by $\Bad(\lambda,\mu)$.

\begin{conj}\label{C1} Let $G$ be a %semisimple
Lie group, $\Gamma$ %an irreducible
a lattice in $G$, $X = \ggm$,  and let $F^+$ be a one-parameter $\Ad$-diagonalizable % non-quasiunipotent
subsemigroup of $G$. Then   $E(F^+)$ is HAW on $X$. Consequently, the conclusion of Theorem \ref{T:dim} holds in such generality.
\end{conj}

%establishes it under an additional assumption that $F^+$ is diagonalizable; the case of non-trivial unipotent part is not covered.
%Using the argument presented in \S\ref{S:6},
%it would be possible to derive Conjecture \ref{C1} from the following conjectural generalization of Theorem \ref{T:HAW_U}:
Similarly to what is done in \S\ref{S:6}, it would be possible to derive Conjecture \ref{C1} from the following conjectural generalization of Theorem~\ref{T:HAW_U}:

\begin{conj}\label{C2} Let $G$, $\Gamma$, $X$ and   $F^+$ be as in Conjecture \ref{C1}, and let $H = H(F^+)$ be as in \equ{defehs}. Then the set  \eq{efplush}{\{h \in H: h\Gamma \in E(F^+)\}} is HAW on $H$.
\end{conj}

If $G$ has real rank $1$, the validity of the above conjectures can be extracted from the work of Dani \cite{Da2}.
The main results of the present paper  (Theorems  \ref{T:HAW} and  \ref{T:HAW-expand}) take care of the case \equ{sld}.
Note also that
%, in view of \cite[Theorem 1.3]{BFS} and \cite{Da1},
 the conclusion of Conjecture \ref{C2}, and hence of Conjecture \ref{C1}, holds for
$$G = \SL_{m +n }( \RR), \ \Gamma = \SL_{m +n }( \ZZ), \ X = \ggm$$
and
$${g_t = \diag(e^{t/m }, \ldots,
e^{ t/m }, e^{-t/n }, \ldots,
e^{-t/n })%\in G
}$$
for any $m,n\in\NN$. Then
$$H = H(F^+) = \left\{
\begin{pmatrix}
   I_m & A \\
    0 & I_n
       \end{pmatrix}: A\in M_{m\times n}\right\},$$ and bounded $g_t$-trajectories correspond to (unweighted) badly approximable systems of $m$ linear forms in $n$ variables. In this case the set  \equ{efplush} was proved to be winning by Schmidt \cite{Sc2}, and later its HAW property was established by Broderick, Fishman and Simmons \cite{BFS}.

%\vfil\eject
Finally, let us mention that  recently the second-named author jointly with  Yu  \cite{GY} have proved that, for the special case \eq{spcase}{r_1=\cdots=r_{d-1} \ge r_d\,,}  $\Bad(\rr)$ is HAW. In other words,
%the set  \equ{efplush} is HAW on $H$, where $G$ and $\Gamma$ are as in  \equ{sldgeneral}, $F^+$ is as in  \equ{sldplusgt} and  \equ{spcase},  and $H$ is a proper subgroup of $H(F^+)$.
the set  \equ{efplush} is HAW on a proper subgroup $H$ of $H(F^+)$, where $G$ and $\Gamma$ are as in  \equ{sldgeneral} and $F^+$ is as in  \equ{sldplusgt} and  \equ{spcase}. This raises a hope that Conjectures  \ref{C2}  and   \ref{C1} can be verified at least for the above special  case.

%Beyond the case   \equ{sld}, the only available result on the intersection of sets $$E(F^+_1)$

 %\footnote{The scope of the conjecture was much more general, not restricted to the case \equ{sld}.}

\end{document}